\documentclass[reqno,a4paper]{amsart}
\usepackage{amssymb,setspace}
\usepackage{ifpdf}
\ifpdf
  \usepackage[hyperindex]{hyperref}
\else
  \expandafter\ifx\csname dvipdfm\endcsname\relax
    \usepackage[hypertex,hyperindex]{hyperref}
  \else
    \usepackage[dvipdfm,hyperindex]{hyperref}
  \fi
\fi
\theoremstyle{plain}
\newtheorem{thm}{Theorem}[section]

\theoremstyle{remark}
\newtheorem{rem}{Remark}[section]
\numberwithin{equation}{section}
\DeclareMathOperator{\re}{Re}
\DeclareMathOperator{\td}{d\mspace{-2mu}}
\date{This manuscript commenced on Saturday 17 April 2010 and completed on Thursday 23 September 2010 in Tianjin City, China}
\allowdisplaybreaks[4]

\begin{document}

\title[A completely monotonic function involving tri-gamma functions]
{A completely monotonic function involving the gamma and tri-gamma functions}

\author[F. Qi]{Feng Qi}
\address{Department of Mathematics, School of Science, Tianjin Polytechnic University, Tianjin City, 300387, China}
\email{\href{mailto: F. Qi <qifeng618@gmail.com>}{qifeng618@gmail.com}, \href{mailto: F. Qi <qifeng618@hotmail.com>}{qifeng618@hotmail.com}, \href{mailto: F. Qi <qifeng618@qq.com>}{qifeng618@qq.com}}
\urladdr{\url{http://qifeng618.wordpress.com}}

\begin{abstract}
In this paper we provide necessary and sufficient conditions on $a$ for the function
$
\frac{1}{2}\ln(2\pi)-x+\bigl(x-\frac{1}{2}\bigr)\ln x-\ln\Gamma(x)+\frac1{12}{\psi'(x+a)}
$
and its negative to be completely monotonic on $(0,\infty)$, where $a\ge0$ is a real number, $\Gamma(x)$ is the classical gamma function, and $\psi(x)=\frac{\Gamma'(x)}{\Gamma(x)}$ is the di-gamma function. As applications, some known results and new inequalities are derived.
\end{abstract}



\thanks{This paper was typeset using \AmS-\LaTeX}

\maketitle

\section{Introduction}
It is well known that the classical Euler's gamma function may be defined by
\begin{equation}\label{gamma-dfn}
\Gamma(x)=\int^\infty_0t^{x-1} e^{-t}\td t
\end{equation}
for $x>0$, that the logarithmic derivative of $\Gamma(x)$ is called psi or di-gamma function and denoted by
\begin{equation}
  \psi(x)=\frac{\td}{\td x}\ln\Gamma(x)=\frac{\Gamma'(x)}{\Gamma(x)}
\end{equation}
for $x>0$, that the derivatives $\psi'(x)$ and $\psi''(x)$ for $x>0$ are respectively called tri-gamma and tetra-gamma functions, and that the derivatives $\psi^{(i)}(x)$ for $i\in\mathbb{N}$ and $x>0$ are called polygamma functions.
\par
We recall from~\cite[Chapter~XIII]{mpf-1993} and~\cite[Chapter~IV]{widder} that a function $f(x)$ is said to be completely monotonic on an interval $I$ if it has derivatives of all orders on $I$ and satisfies
\begin{equation}\label{CM-dfn}
0\le(-1)^{n}f^{(n)}(x)<\infty
\end{equation}
for $x\in I$ and all integers $n\ge0$. If $f(x)$ is non-constant, then the inequality~\eqref{CM-dfn} is strict (see~\cite[p.~98]{Dubourdieu96} or~\cite[p.~82]{e-gam-rat-comp-mon}). The class of completely monotonic functions may be characterized by the celebrated Bernstein-Widder Theorem \cite[p.~160, Theorem~12a]{widder} which reads that a necessary and sufficient condition that $f(x)$ should be completely monotonic in $0\le x<\infty$ is that
\begin{equation} \label{berstein-1}
f(x)=\int_0^\infty e^{-xt}\td\alpha(t),
\end{equation}
where $\alpha(t)$ is bounded and non-decreasing and the integral converges for $0\le x<\infty$.
\par
For $x\in(0,\infty)$ and $a\ge0$, let
\begin{equation}\label{F-a(x)}
F_a(x)=\ln\Gamma(x)- \biggl(x-\frac12\biggr)\ln x-\frac1{12}\psi'(x+a).
\end{equation}
In~\cite[Theorem~1]{Merkle-rocky} M. Merkle proved that the function $F_0(x)$ is strictly concave and the function $F_a(x)$ for $a\geq{\frac12}$ is strictly convex on $(0,\infty)$. This was surveyed and reviewed in~\cite[p.~46, Section~4.3.3]{bounds-two-gammas.tex}.
\par
In recent years, some new results on the complete monotonicity of functions involving the gamma and polygamma functions have been obtained in~\cite{psi-proper-fraction-degree-two.tex, mon-funct-gamma-unit-ball.tex, unit-ball.tex, Yang-Fan-2008-Dec-simp.tex, Guo-Qi(201008).tex, property-psi-ii-final.tex, Lv-Sun-Chu-JIA-2011, QiBerg.tex, BAustMS-5984-RV.tex, SCM-2012-0142.tex, Open-TJM-2003-Banach.tex, Sriv-Guo-Qi-Computer-016, Zhao-Chu-Wang-AAA-2011, x-4-di-tri-gamma-upper-lower-combined.tex}, for example.
\par
The aims of this paper are to generalize the convexity of the function $F_a(x)$ and to derive known results and some new inequalities.

\section{Complete monotonicity}

The first aim of this paper is to generalize the convexity of $F_a(x)$ to complete monotonicity which may be stated as Theorem~\ref{Merkle-Convexity2Complete-Mon-thm1} below.

\begin{thm}\label{Merkle-Convexity2Complete-Mon-thm1}
For $x\in(0,\infty)$ and $a\ge0$, let
\begin{equation}
f_a(x)=\frac{1}{2}\ln(2\pi)-x+\biggl(x-\frac{1}{2}\biggr)\ln x-\ln\Gamma(x)+\frac1{12}{\psi'(x+a)}.
\end{equation}
Then the functions $f_0(x)$ and $-f_a(x)$ for $a\ge\frac12$ are completely monotonic on $(0,\infty)$.
\end{thm}

\begin{proof}
Using recursion formulas $\Gamma(x+1)=x\Gamma(x)$ and
\begin{equation}\label{psi'(x+1)-psi'(x)=-frac1-x2}
\psi'(x+1)-\psi'(x)=-\frac1{x^2}
\end{equation}
for $x>0$, an easy calculation yields
\begin{align*}
f_a(x)-f_a(x+1)&=1+\biggl(x+\frac{1}{2}\biggr)\ln\biggl(\frac{x}{x+1}\biggr) +\frac1{12}\bigl[\psi'(x+a)-\psi'(x+a+1)\bigr]\\
&=1+\biggl(x+\frac{1}{2}\biggr)\ln\biggl(\frac{x}{x+1}\biggr)+\frac1{12(x+a)^2}
\end{align*}
and
\begin{equation*}
[f_a(x)-f_a(x+1)]'=\frac{1}{2(x+1)}+\frac{1}{2x}-\frac{1}{6(a+x)^3} +\ln\biggl(\frac{x}{x+1}\biggr).
\end{equation*}
Utilizing formulas
\begin{equation}\label{Gamma(z)=k-z-int}
\Gamma(z)=k^z\int_0^\infty t^{z-1}e^{-kt}\td t
\end{equation}
and
\begin{equation}\label{ln-frac}
\ln\frac{b}a=\int_0^\infty\frac{e^{-au}-e^{-bu}}u\td u
\end{equation}
for $\re z>0$, $\re k>0$, $a>0$ and $b>0$, see~\cite[p.~255, 6.1.1 and p.~230, 5.1.32]{abram}, gives
\begin{equation}
\begin{split}\label{phi_a(t)-dfn}
[f_a(x)-f_a(x+1)]'&=\int_0^\infty\biggl[\frac12e^{-t}+\frac12 -\frac1{12}t^2e^{-at} +\frac{e^{-t}-1}t\biggr]e^{-xt}\td t\\
&\triangleq\int_0^\infty\phi_a(t)e^{-xt}\td t.
\end{split}
\end{equation}
It is easy to see that
$$
\phi_0(t)=-\frac{(t^3-6t+12) e^t-6 (t+2)}{12 te^t} =-\frac1{12e^t}\sum_{i=4}^\infty\frac{(i-3)\bigl(i^2-4\bigr)}{i!}t^{i-1}<0
$$
and
\begin{align*}
\phi_{1/2}(t)&=\frac{6(t-2)e^t-t^3e^{t/2}+6(t+2)}{12te^t} \\*
&=\frac1{12e^t}\sum_{i=5}^\infty\frac{(i-2)\bigl(3\cdot2^{i-2}-i^2+i\bigr)}{i!\cdot2^{i-3}}t^{i-1}\\*
&>0
\end{align*}
on $(0,\infty)$, where the inequality $3\cdot2^{i-2}-i^2+i>0$ for $i\ge5$ may be verified by induction. As a result, the function
$$
[f_0(x+1)-f_0(x)]'=f_0'(x+1)-f_0'(x)
$$
and
$$
[f_{1/2}(x)-f_{1/2}(x+1)]'=f_{1/2}'(x)-f_{1/2}'(x+1)
$$
are completely monotonic on $(0,\infty)$, that is,
$$
(-1)^k[f_0'(x+1)-f_0'(x)]^{(k)}=(-1)^kf_0^{(k+1)}(x+1)-(-1)^kf_0^{(k+1)}(x)\ge0
$$
and
$$
(-1)^k[f_{1/2}'(x)-f_{1/2}'(x+1)]^{(k)}=(-1)^kf_{1/2}^{(k+1)}(x)-(-1)^kf_{1/2}^{(k+1)}(x+1)\ge0
$$
for $k\ge0$. By induction, we have
\begin{multline}\label{f-0-induction}
(-1)^kf_0^{(k+1)}(x)\le(-1)^kf_0^{(k+1)}(x+1)\le (-1)^kf_0^{(k+1)}(x+2)\\* \le(-1)^kf_0^{(k+1)}(x+3)\le\dotsm\le(-1)^k\lim_{m\to\infty}f_0^{(k+1)}(x+m)
\end{multline}
and
\begin{multline}\label{f-1-2-induction}
(-1)^kf_{1/2}^{(k+1)}(x)\ge (-1)^kf_{1/2}^{(k+1)}(x+1)\ge (-1)^kf_{1/2}^{(k+1)}(x+2)\\
\ge(-1)^kf_{1/2}^{(k+1)}(x+3)\ge\dotsm\ge(-1)^k\lim_{m\to\infty}f_{1/2}^{(k+1)}(x+m)
\end{multline}
for $k\ge0$.
\par
It is not difficult to obtain
\begin{equation*}
f_a'(x)=\frac{\psi''(a+x)}{12}-\psi(x)+\ln x-\frac{1}{2x}
\end{equation*}
and
$$
f_a^{(i)}(x)=\frac{\psi^{(i+1)}(a+x)}{12}-\psi^{(i-1)}(x) +\frac{(-1)^i(i-2)!}{x^{i-1}}+\frac{(-1)^i(i-1)!}{2x^i},\quad i\ge2.
$$
In the light of the double inequalities
\begin{equation}\label{psi'ineq}
\ln x-\frac1x<\psi(x) <\ln x-\frac1{2x}
\end{equation}
and
\begin{equation}\label{qi-psi-ineq}
\frac{(i-1)!}{x^i}+\frac{i!}{2x^{i+1}}<(-1)^{i+1}\psi^{(i)}(x) <\frac{(i-1)!}{x^i}+\frac{i!}{x^{i+1}}
\end{equation}
for $x>0$ and $i\in\mathbb{N}$, see~\cite[Lemma~3]{MIA-1729.tex}, \cite[p.~107, Lemma~3]{theta-new-proof.tex-BKMS}, \cite[p.~79]{AAM-Qi-09-PolyGamma.tex} and \cite[Lemma~3]{subadditive-qi-guo-jcam.tex}, we immediately derive
$$
\lim_{x\to\infty}f_a'(x)=0\quad \text{and}\quad \lim_{x\to\infty}f_a^{(i)}(x)=0
$$
for $i\ge2$ and $a\ge0$. Combining this with~\eqref{f-0-induction} and~\eqref{f-1-2-induction}, we deduce
\begin{equation}\label{qi-psi-ineq-and}
(-1)^kf_0^{(k+1)}(x)\le0 \quad \text{and}\quad (-1)^kf_{1/2}^{(k+1)}(x)\ge0
\end{equation}
for $k\ge0$ on $(0,\infty)$.
\par
From the formula
\begin{equation}\label{p.-258-6.1.50-abram}
\ln\Gamma(z)=\biggl(z-\frac{1}{2}\biggr)\ln z-z+\frac{1}{2}\ln(2\pi) +2\int_0^\infty\frac{\arctan(t/z)}{e^{2\pi t}-1}\td t
\end{equation}
for $\re z>0$, see~\cite[p.~258, 6.1.50]{abram}, and the double inequality~\eqref{qi-psi-ineq} for $i=1$, we easily obtain
\begin{equation}\label{lim-f-a(x)=0}
\lim_{x\to\infty}f_a(x)=0
\end{equation}
for $a\ge0$. Inequalities in~\eqref{qi-psi-ineq-and} imply that the functions $-f_0(x)$ and $f_{1/2}(x)$ are increasing on $(0,\infty)$. Hence, we have
\begin{equation}\label{qi-psi-ineq-and-and}
f_0(x)>0 \quad\text{and}\quad f_{1/2}(x)<0
\end{equation}
on $(0,\infty)$.
\par
From~\eqref{qi-psi-ineq-and} and~\eqref{qi-psi-ineq-and-and}, we conclude that the functions $f_0(x)$ and $-f_{1/2}(x)$ are completely monotonic on $(0,\infty)$.
\par
It is clear that
$$
-f_a(x)=-f_{1/2}(x)+\frac1{12}\biggl[\psi'\biggl(x+\frac12\biggr)-\psi'(x+a)\biggr].
$$
From the facts that the tri-gamma function
\begin{equation}\label{psim-trigamma}
\psi'(x)=\int_{0}^{\infty}\frac{t}{1-e^{-t}}e^{-xt}\td t
\end{equation}
for $x>0$, see~\cite[p.~260, 6.4.1]{abram}, is completely monotonic on $(0,\infty)$, that the difference $f(x)-f(x+\alpha)$ for any given real number $\alpha>0$ of any completely monotonic function $f(x)$ on $(0,\infty)$ is also completely monotonic on $(0,\infty)$, and that the sum of finitely many completely monotonic functions on an interval $I$ is still completely monotonic on $I$, it readily follows that the function $-f_a(x)$ for $a>\frac12$ is also completely monotonic on $(0,\infty)$. The proof of Theorem~\ref{Merkle-Convexity2Complete-Mon-thm1} is complete.
\end{proof}

\section{Necessary and sufficient conditions}

The second aim of this paper is to answer a natural problem: Find the best constants $\alpha\ge0$ and $\beta\le\frac12$ such that $f_\alpha(x)$ and $-f_\beta(x)$ are both completely monotonic on $(0,\infty)$.

\begin{thm}\label{Merkle-Convexity2Complete-Mon-thm2}
The function $f_\alpha(x)$ is completely monotonic on $(0,\infty)$ if and only if $\alpha=0$, and so is the function $-f_\beta(x)$ if and only if $\beta\ge\frac12$.
\end{thm}

\begin{proof}[The first proof]
The conclusion that the function $\phi_a(t)$ defined in~\eqref{phi_a(t)-dfn} is positive or negative on $(0,\infty)$ is equivalent to
\begin{equation}\label{eq.(3.1)}
a\gtrless-\frac{1}{t}\ln\biggl[\frac{12}{t^2}\biggl(\frac{e^{-t}+1}2+\frac{e^{-t}-1}t\biggr)\biggr] \triangleq-\varphi(t)=-\frac1t\ln\varphi_1(t),\quad t>0.
\end{equation}
By L'H\^ospital rule, we have
$$
\lim_{t\to0^+}\varphi_1(t)=6\lim_{t\to0^+}\frac{t(e^{-t}+1)+2(e^{-t}-1)}{t^3}
=2\lim_{t\to0^+}\frac{e^{-t}(e^t-t-1)}{t^2}=1
$$
and $\lim_{t\to\infty}\varphi_1(t)=0$. Hence, the function $\varphi(t)$ can be represented as
\begin{multline*}
\varphi(t)=\frac{\ln\varphi_1(t)-\ln\varphi_1(0)}t =\frac1t\int_0^t\frac{\varphi_1'(u)}{\varphi_1(u)}\td u \\*
=-\frac1t\int_0^t\frac{2(u-3)e^u+u^2+4u+6}{u[(u-2)e^u+u+2]}\td u\triangleq-\frac1t\int_0^t\varphi_2(u)\td u
\end{multline*}
for $t>0$. Since
\begin{align*}
\varphi_2'(u)&=-\frac{2 \bigl(u^2-6 u+6\bigr)e^{2 u}+\bigl(u^4+8 u^2-24\bigr)e^u +2\bigl(u^2+6u+6\bigr)}{u^2[(u-2)e^u+u+2]^2}\\
&=-\frac1{u^2[(u-2)e^u+u+2]^2}\\
&\quad\times\sum_{i=8}^\infty \frac{2^{i-1}\bigl(i^2-13 i+24\bigr)+i^4-6 i^3+19 i^2-14i-24}{i!}u^i\\
&<0
\end{align*}
for $u>0$, where
\begin{align*}
&\quad2^{i-1}\bigl(i^2-13 i+24\bigr)+i^4-6 i^3+19 i^2-14i-24\\
&=(1+1)^{i-1}\bigl(i^2-13 i+24\bigr)+i^4-6 i^3+19 i^2-14i-24\\
&>i\bigl(i^2-13 i+24\bigr)+i^4-6 i^3+19 i^2-14i-24\\
&=(i-5)i^3+6 i^2+2(5i-12)\\
&>0
\end{align*}
for $i\ge8$, the function $\varphi_2(u)$ is strictly decreasing on $(0,\infty)$, and, by~\cite[p.~418, Lemma 2.1]{qx3} and \cite[p.~3260, (1.6)]{qizh}, the arithmetic mean
$$
-\varphi(t)=\frac1t\int_0^t{\varphi_2(u)}\td u
$$
is strictly decreasing, and $\varphi(x)$ is strictly increasing, on $(0,\infty)$. From L'H\^ospital rule and limits
$$
\lim_{u\to0^+}\varphi_2(u)=\frac12\quad \text{and}\quad \lim_{u\to\infty}\varphi_2(u)=0,
$$
we obtain
$$
\lim_{t\to0^+}\varphi(t)=-\frac12\quad \text{and}\quad \lim_{t\to\infty}\varphi(t)=0.
$$
As a result, from~\eqref{eq.(3.1)}, it follows that
\begin{enumerate}
\item
when $a=0$, the function $[f_a(x+1)-f_a(x)]'$ is completely monotonic;
\item
when $a\ge\frac12$, the function $[f_a(x)-f_a(x+1)]'$ is completely monotonic.
\end{enumerate}
Along with the corresponding argument as in the proof of Theorem~\ref{Merkle-Convexity2Complete-Mon-thm1}, we obtain that the sufficient condition for $f_a(x)$ or $-f_a(x)$ to be completely monotonic on $(0,\infty)$ is $a=0$ or $a\ge\frac12$ respectively.
\par
Conversely, if $-f_a(x)$ is completely monotonic on $(0,\infty)$, then $f_a(x)$ is increasing and negative on $(0,\infty)$, so $x^2f_a(x)<0$ on $(0,\infty)$. From the double inequality
\begin{gather}\label{lem2}
\frac{1}{2x^2}-\frac{1}{6x^3} <\frac{1}{x}-\psi'(x+1) <\frac{1}{2x^2}-\frac{1}{6x^3}+\frac{1}{30x^5}
\end{gather}
on $(0,\infty)$, see~\cite[p.~305, Corollary~1]{qi-cui-jmaa}, it is easy to see that
\begin{equation}\label{frac12-infinity}
\lim_{x\to\infty}\biggl\{x^2\biggl[\psi'(x)-\frac1x\biggr]\biggr\}=\frac12.
\end{equation}
Using the asymptotic formula
\begin{gather}
\ln\Gamma(z)\sim\biggl(z-\frac12\biggr)\ln z-z+\frac12\ln(2\pi)+\frac1{12z}-\frac1{360z^3} +\frac{1}{1260z^5}-\dotsm
\end{gather}
as $z\to\infty$ in $\lvert \arg z\rvert<\pi$, see~\cite[p.~257, 6.1.41]{abram}, gives
\begin{multline}\label{frac(a)(12)-infinity}
\lim_{x\to\infty}\biggl\{x^2\biggl[\frac{1}{2}\ln(2\pi)-x +\biggl(x-\frac{1}{2}\biggr) \ln x -\ln\Gamma(x)+\frac1{12(x+a)}\biggr]\biggr\}\\*
=\lim_{x\to\infty}\biggl\{x^2\biggl[\frac1{12(x+a)}-\frac1{12x}+O\biggl(\frac1{x^2}\biggr)\biggr]\biggr\}
=-\frac{a}{12}.
\end{multline}
In virtue of~\eqref{frac12-infinity} and~\eqref{frac(a)(12)-infinity}, we obtain
\begin{multline*}
x^2f_a(x)=x^2\biggl[\frac{1}{2}\ln(2\pi)-x +\biggl(x-\frac{1}{2}\biggr) \ln x-\ln\Gamma(x) +\frac1{12(x+a)}\biggr]\\*
+\frac{x^2}{12}\biggl[{\psi'(x+a)}-\frac1{x+a}\biggr] \to-\frac{a}{12}+\frac1{12}\cdot\frac12
\end{multline*}
as $x$ tends to $\infty$. So the necessary condition for $-f_a(x)$ to be completely monotonic on $(0,\infty)$ is $a\ge\frac12$.
\par
If $f_a(x)$ for $a>0$ is completely monotonic on $(0,\infty)$, then $f_a(x)$ should be decreasing and positive on $(0,\infty)$, but utilizing~\eqref{p.-258-6.1.50-abram} leads to
\begin{align*}
\lim_{x\to0^+}f_a(x)&=\lim_{x\to0^+}\biggl[\frac{1}{2}\ln(2\pi)-x+\biggl(x-\frac{1}{2}\biggr)\ln x-\ln\Gamma(x)\biggr]+\frac1{12}{\psi'(a)}\\
&=\frac1{12}{\psi'(a)}-2\lim_{x\to0^+}\int_0^\infty\frac{\arctan(t/x)}{e^{2\pi t}-1}\td t\\
&=\frac1{12}{\psi'(a)}-\pi\int_0^\infty\frac1{e^{2\pi t}-1}\td t\\
&=-\infty
\end{align*}
which leads to a contradiction. So the necessary condition for $f_a(x)$ to be completely monotonic on $(0,\infty)$ is $a=0$. The first proof of Theorem~\ref{Merkle-Convexity2Complete-Mon-thm2} is complete.
\end{proof}

\begin{proof}[The second proof]
The famous Binet's first formula of $\ln\Gamma(x)$ for $x>0$ is given by
\begin{equation}\label{ebinet}
\ln \Gamma(x)= \left(x-\frac{1}{2}\right)\ln x-x+\ln \sqrt{2\pi}\,+\theta(x),
\end{equation}
where
\begin{equation}\label{ebinet1}
\theta(x)=\int_{0}^{\infty}\biggl(\frac{1}{e^{t}-1}-\frac{1}{t}
+\frac{1}{2}\biggr)\frac{e^{-xt}}{t}\td t
\end{equation}
for $x>0$ is called the remainder of Binet's first formula for the logarithm of the gamma function $\Gamma(x)$. See \cite[p.~11]{magnus} or \cite[p.~462]{Extended-Binet-remiander-comp.tex}. Combining this with the integral representation
\begin{equation}\label{psim}
\psi ^{(k)}(x)=(-1)^{k+1}\int_{0}^{\infty}\frac{t^{k}}{1-e^{-t}}e^{-xt}\td t
\end{equation}
for $x>0$ and $k\in\mathbb{N}$, see~\cite[p.~260, 6.4.1]{abram}, yields
\begin{equation}
\begin{split}\label{tri-psi-theta-eq}
  f_a(x)&=\frac1{12}\psi'(x+a)-\theta(x) \\
  &=\int_{0}^{\infty}\biggl[\frac{te^{(1-a)t}}{12(e^{t}-1)}-\frac1{t}\biggl(\frac{1}{e^{t}-1} -\frac{1}{t}+\frac{1}{2}\biggr)\biggr]{e^{-xt}}\td t.
\end{split}
\end{equation}
It is not difficult to see that the positivity and negativity are equivalent to
\begin{equation}\label{q(t)-dfn-eq}
  a\lessgtr-\frac1t\ln\biggl[\frac{12(e^t-1)}{t^2e^t}\biggl(\frac{1}{e^{t}-1} -\frac{1}{t}+\frac{1}{2}\biggr)\biggr] =-\varphi(t),
\end{equation}
where $\varphi(t)$ is defined by~\eqref{eq.(3.1)}. The rest proof is the same as in the first proof of Theorem~\ref{Merkle-Convexity2Complete-Mon-thm2}. The second proof of Theorem~\ref{Merkle-Convexity2Complete-Mon-thm2} is complete.
\end{proof}

\section{Remarks}
In this section, we list more results in the form of remarks.

\begin{rem}\label{rem-merkle-mon-4.1}
From proofs of Theorem~\ref{Merkle-Convexity2Complete-Mon-thm1} and Theorem~\ref{Merkle-Convexity2Complete-Mon-thm2}, we can abstract a general and much useful conclusion below.

\begin{thm}
A function $f(x)$ defined on an infinite interval $I$ tending to $\infty$ is completely monotonic if and only if
\begin{enumerate}
  \item
  there exist positive numbers $\alpha_i$ such that the differences
\begin{equation}
  (-1)^i[f(x)-f(x+\alpha_i)]^{(i)}
\end{equation}
  are nonnegative for all integers $i\ge0$ on $I$;
  \item
  the limits
\begin{equation}
  \lim_{x\to\infty}\bigl[(-1)^if^{(i)}(x)\bigr]=a_i\ge0
\end{equation}
  exist for all integers $i\ge0$.
\end{enumerate}
\end{thm}

This is essentially a generalization of \cite[p.~126, Lemma~2.2]{subadditive-qi.tex}, \cite[Lemma~1]{notes-best-simple-open-jkms.tex}, \cite[p.~223, Lemma~2.1]{property-psi.tex}, \cite[p.~107, Lemma~4]{theta-new-proof.tex-BKMS}, \cite[p.~526, Lemma~2.1]{notes-best-simple-equiv.tex}, \cite[p.~1981, Lemma~2.5]{notes-best-simple-cpaa.tex}, and~\cite[p.~2155, Lemma~1]{subadditive-qi-guo-jcam.tex}, which was also implicitly applied in~\cite{BNGuo-FQi-HMSrivastava2010.tex, AAM-Qi-09-PolyGamma.tex}, for example.
\end{rem}

\begin{rem}\label{rem-sec-2.1}
Because
\begin{equation}
f_a(x)=\frac12\ln(2\pi)-x-F_a(x)
\end{equation}
and $f_a''(x)=-F_a''(x)$ on $(0,\infty)$, the concavity of $F_0(x)$ and the convexity of $F_a(x)$ for $a\ge\frac12$ obtained in~\cite[Theorem~1]{Merkle-rocky} can be concluded readily from the complete monotonicity of $f_\alpha(x)$ and $-f_\beta(x)$ on $(0,\infty)$ established in Theorems~\ref{Merkle-Convexity2Complete-Mon-thm1} and~\ref{Merkle-Convexity2Complete-Mon-thm2}.
\end{rem}

\begin{rem}
We also recall from \cite{Atanassov, minus-one, auscmrgmia} that a function $f(x)$ is said to be logarithmically completely monotonic on an interval $I$ if it has derivatives of all orders on $I$ and its logarithm $\ln f(x)$ satisfies
\begin{equation}\label{lcm-dfn}
0\le(-1)^k[\ln f(x)]^{(k)}<\infty
\end{equation}
for all integers $k\ge1$ on $I$. It was proved once again in \cite{CBerg, absolute-mon-simp.tex, compmon2, minus-one, schur-complete} that logarithmically completely monotonic functions on an interval $I$ must be completely monotonic on $I$, but not conversely. By the way, the preprint \cite{minus-one} were extended to, divided into, and formally published as \cite{compmon2, minus-one-JKMS.tex}, the preprint~\cite{auscmrgmia} was brought out as~\cite{e-gam-rat-comp-mon}, and the preprint~\cite{schur-complete} was modified and split into~\cite{absolute-mon-simp.tex, Mon-Two-Seq-AMEN.tex}. For more information on the history and properties of logarithmically completely monotonic functions, please refer to~\cite{Atanassov, CBerg}, \cite[pp.~21\nobreakdash--23]{absolute-mon-simp.tex}, \cite[pp.~5\nobreakdash--6, Section~1.5]{bounds-two-gammas.tex}, \cite[pp.~2154\nobreakdash--2155, Remark~8]{subadditive-qi-guo-jcam.tex} and closely related references therein.
\par
For $a\ge0$, let
\begin{equation}
  g_a(x)=-\ln\Gamma(x)+\biggl(x-\frac12\biggr)\ln x-x+\frac1{12}\psi'(x+a),\quad x>0.
\end{equation}
It is obvious that
\begin{equation}\label{f-a--g-a(x)}
f_a(x)=\frac12\ln(2\pi)+g_a(x)
\end{equation}
on $(0,\infty)$ for $a\ge0$, with the limit~\eqref{lim-f-a(x)=0}. It is not difficult to see that Theorem~1 in~\cite[pp.~338\nobreakdash--341]{Alzer1} may be reworded as follows: For $0<s<1$ the function
$
  \exp[{g_a(x+s)}-{g_a(x+1)}]
$
is logarithmically completely monotonic on $(0,\infty)$ if and only if $a\ge\frac12$, and so is the function
$
  \exp[{g_a(x+1)}-{g_a(x+s)}]
$
if and only if $a=0$. This was reviewed in~\cite[pp.~38\nobreakdash--39, Section~3.18]{bounds-two-gammas.tex}.
\par
In virtue of complete monotonicity of $f_a(x)$ and Remark~\ref{rem-merkle-mon-4.1}, it follows that the difference
$f_a(x+s)-f_a(x+t)=g_a(x+s)-g_a(x+t)$
for $t>s$ and $a\ge0$ is completely monotonic with respect to $x\in(-s,\infty)$ if and only if $a=0$, and so is its negative if and only if $\alpha\ge\frac12$. Therefore, by the second item of Theorem~5 in~\cite[p.~1286]{minus-one-JKMS.tex}, it follows that the function
$
  \exp[{f_a(x+s)}-{f_a(x+t)}]
$
for $t>s$ and $a\ge0$ is logarithmically completely monotonic with respect to $x$ on $(-s,\infty)$ if and only if $a\ge\frac12$, and so is the function
$
  \exp[{f_a(x+t)}-{f_a(x+s)}]
$
if and only if $a=0$. In other words, the function
\begin{equation}\label{li-ext-fun}
\frac{\Gamma(x+s)}{\Gamma(x+t)}\cdot\frac{(x+t)^{x+t-1/2}}{(x+s)^{x+s-1/2}}
\exp\biggl[s-t+\frac{\psi'(x+t+\alpha) -\psi'(x+s+\alpha)}{12}\biggr]
\end{equation}
for $s<t$ and $\alpha\ge0$ is logarithmically completely monotonic with respect to $x\in(-s,\infty)$ if and only if $\alpha\ge\frac12$, and so is the reciprocal of~\eqref{li-ext-fun} if and only if $\alpha=0$.
\par
The monotonicity of~\eqref{li-ext-fun} and its reciprocal implies that the double inequality
\begin{multline}\label{li-ext-fun-ineq}
\exp\biggl[t-s+\frac{\psi'(x+s+\beta) -\psi'(x+t+\beta)}{12}\biggr] \le\frac{\Gamma(x+s)}{\Gamma(x+t)}\cdot\frac{(x+t)^{x+t-1/2}}{(x+s)^{x+s-1/2}}\\*
\le \exp\biggl[t-s+\frac{\psi'(x+s+\alpha) -\psi'(x+t+\alpha)}{12}\biggr]
\end{multline}
for $\alpha>\beta\ge0$, $s<t$ and $x\in(-s,\infty)$ is valid if and only if $\beta=0$ and $\alpha\ge\frac12$.
\par
In conclusion, from complete monotonicity in Theorem~\ref{Merkle-Convexity2Complete-Mon-thm2}, Theorem~1 and Corollary~1 in~\cite{Alzer1} together with Theorem~3 and its corollary in~\cite{laj-7.pdf} may be deduced and extended straightforwardly. This means that Theorem~\ref{Merkle-Convexity2Complete-Mon-thm2} is stronger not only than~\cite[Theorem~1]{Alzer1} but also than~\cite[Theorem~3]{laj-7.pdf}.
\end{rem}

\begin{rem}
As a consequence of Theorem~\ref{Merkle-Convexity2Complete-Mon-thm2}, the following double inequality is easily obtained: For $x>0$, the double inequality
\begin{equation}
\sqrt{2\pi}\,x^{x-1/2}\exp\biggl[\frac{\psi'(x+\beta)}{12}-x\biggr]<\Gamma(x) <\sqrt{2\pi}\,x^{x-1/2}\exp\biggl[\frac{\psi'(x+\alpha)}{12}-x\biggr]
\end{equation}
is valid if and only if $\alpha=0$ and $\beta\ge\frac12$.
\par
For more inequalities for bounding the gamma function $\Gamma(x)$, please refer to~\cite{theta-new-proof.tex-BKMS, note-on-li-chen.tex, refine-Ivady-gamma-PMD.tex} and closely related references therein.
\end{rem}

\begin{rem}
The equation~\eqref{tri-psi-theta-eq} in the second proof of Theorem~\ref{Merkle-Convexity2Complete-Mon-thm2} tells us the integral representations of the completely monotonic functions $f_0(x)$ and $-f_a(x)$ for $a\ge\frac12$.
\end{rem}

\begin{rem}
For the history, background, motivations, and recent developments of this topic, please refer to the survey and expository paper~\cite{bounds-two-gammas.tex, Wendel-Gautschi-type-ineq-Banach.tex} and plenty of references therein.
\end{rem}

\end{document}